\documentclass[12pt]{article}
\usepackage{amsthm,amsmath,amssymb}
\usepackage[active]{srcltx}
\newtheorem{theorem}{Theorem}%[section]
\newtheorem{lemma}[theorem]{Lemma}%[section]

\theoremstyle{remark}
\newtheorem{remark}[theorem]{Remark}%[section]
\newtheorem{example}[theorem]{Example}%[section]

\renewcommand{\le}{\leqslant}
\renewcommand{\ge}{\geqslant}

\newcommand{\F}{\mathbb F}
\DeclareMathOperator{\Hom}{Hom}
\DeclareMathOperator{\Ker}{Ker}
\DeclareMathOperator{\im}{Im}

\DeclareMathOperator{\dddots}%
{\!\text{\raisebox{2pt}{$\cdot$}}\cdot
\text{\raisebox{-2pt}{$\cdot$}}}

\begin{document}

\title{Neighborhood radius estimation for Arnold's miniversal deformations of complex and $p$-adic matrices\thanks{Published in Linear Algebra Appl. 512 (2017) 97--112.}}

\author{Victor A. Bovdi\thanks{United Arab Emirates University, Al Ain, UAE, {vbovdi@gmail.com}}
\and
Mohammed A. Salim\thanks{United Arab Emirates University, Al Ain, UAE, {msalim@uaeu.ac.ae}}
\and
Vladimir V. Sergeichuk\thanks{Institute of Mathematics,
Kiev, Ukraine, {sergeich@imath.kiev.ua}}}

\date{}
\maketitle

\begin{abstract}
V.I. Arnold (1971) constructed a
simple normal form to which all complex
matrices $B$ in a neighborhood $U$ of a given square matrix $A$ can be
reduced by similarity transformations
that smoothly depend on the entries of $B$. We calculate the radius of the neighborhood $U$. A.A. Mailybaev  (1999, 2001) constructed a reducing similarity transformation in the form of Taylor series; we construct this transformation by another method.
We extend Arnold's normal form to matrices
over the field $\mathbb Q_p$ of $p$-adic numbers and the field $\F((T))$ of Laurent series over a field $\F$.


\end{abstract}

\section{Introduction}

The reduction of a complex matrix to
its Jordan form is an unstable
operation: both the Jordan form and a
reduction transformation depend
discontinuously on the entries of the
original matrix. Arnold \cite{arn} (see
also \cite{arn2,arn3}) constructed a
\emph{miniversal deformation} of a
square complex matrix $A$; i.e., a
simple normal form $B_{\text{arn}}$ to which all complex
matrices $B$ close to $A$ can be
reduced by similarity transformations
that smoothly depend on the entries of
$B$.

More precisely: Arnold supposes without restriction that $A$ is a Jordan canonical matrix and reduces all matrices $B$ in a neighborhood $U$ of $A$ to the form $B_{\text{arn}}$ by a smooth similarity transformation that acts identically on $A$.
Klimenko and Sergeichuk \cite{k-s} described this reduction in detail.

Many applications of Arnold's normal form in different areas of mathematics are given in more than 120 articles that cite \cite{arn} from the Mathematical Reviews  Citation Database.
We mention only Mailybaev's articles \cite{mai1999,mai2000,mai2001} in which applications of miniversal deformations are based on the fact that the spectrum of $B\in U$ and $B_{\text{arn}}$ coincide but $B_{\text{arn}}$ has a simple form. Mailybaev also constructed a smooth similarity transformation (in the form of Taylor series) that transforms all $B\in U$ to $B_{\text{\rm arn}}$.

Galin \cite{gal} (see also
\cite[\S\,30E]{arn3}) obtained miniversal deformations of real matrices by realification of Arnold's miniversal deformations of complex matrices. Simpler miniversal deformations of real matrices were given by Garcia-Planas and
Sergeichuk \cite{gar-ser}.

The main results of our paper are formulated in Theorem \ref{teo2}:
\begin{itemize}
\item
{\it We extend Arnold's normal form of complex matrices to matrices
over any field that is complete with
respect to a nontrivial absolute value} (in particular, over {\it the field $\mathbb Q_p$ of $p$-adic numbers} and {\it the field $\F((T))$ of Laurent series
over a field $\F$}; see Example \ref{mft}).
We use the Frobenius canonical form for similarity over an arbitrary field instead of the Jordan canonical form.

  \item    {\it Over such a field, we construct a smooth similarity transformation that transforms all $B\in U$ to $B_{\text{\rm arn}}$.} Our method
  differs from the method developed by Mailybaev \cite{mai1999,mai2000,mai2001}.

  \item {\it We give the neighborhood $U$ in an explicit form}, which is important for applications. As far as we know, the estimate of the radius of $U$ was unknown even for complex matrices.

\end{itemize}

Miniversal deformations, reducing transformations, and neighborhood radius estimations for complex matrices under congruence and *congruence were given by Dmytryshyn, Futorny, and Sergeichuk \cite{f_s,def-sesq}.

\section{Preliminaries}
In this section, we recall some definitions and known facts.

\subsection{Arnold's miniversal
deformations of Jordan matrices}\label{sss1}

The similarity class of an $n\times n$ complex $A$ in a small
neighborhood of $A$ can be obtained by
a very small deformation of the affine
matrix space \[\{A+ XA-
AX\,|\,X\in{\mathbb C}^{n\times n}\}\]
since for
each sufficiently small matrix
$X\in{\mathbb C}^{n\times n}$,
\begin{equation}\label{kkd}
\begin{split}
(I_n-&X)^{-1}A(I_n-X)=(I_n+X+X^2+\cdots)A(I_n-X)
\\&=A+(XA-AX)+X(XA-AX)+X^2(XA-AX)+\cdots
\\&=A+(XA-AX)+X(I_n+X+X^2+\cdots)(XA-AX)
\\&=A+\underbrace{(XA-AX)}_{\text{small}}
+\underbrace{X(I_n-X)^{-1}(XA
-AX)}_{\text{very small}}.
\end{split}
\end{equation}
($X$ can be taken small due to the Lipschitz property \cite{rodm}: if $A$ and $B$ are $n\times n$ complex
matrices close to each other and
$B=S^{-1}AS$ with a nonsingular $S$,
then $S$ can be taken near $I_n$.)
The vector space
\begin{equation}\label{eelie}
T(A):=\{XA-AX\,|\,X\in{\mathbb
C}^{n\times n}\}
\end{equation}
is the \emph{tangent space} to the similarity
class of $A$ at the point $A$.
The numbers
\[\dim T(A),\qquad
n^2-\dim T(A)\] are called
the \emph{dimension} and
\emph{codimension}, respectively, of the similarity
class of $A$.

We use the \emph{matrix norm}
$ \|M\|:=\sqrt{\vphantom{\sum^a} \sum |m_{ij}|^2}$ of $M=[m_{ij}]$, the Jordan blocks
\begin{equation}\label{jyc}
J_m(\lambda ):=\begin{bmatrix}
\lambda&1&&0  \\&\lambda& \ddots\\&&\ddots&1\\
0&&&\lambda
               \end{bmatrix}\qquad
               (m\text{-by-}m,\ \lambda\in
               \mathbb C),
\end{equation}
and the matrices
\begin{equation*}\label{5.1}
0^{\downarrow}:=\left[\begin{tabular}{c}
\Large 0 \\[-1mm]
$\!\!* \cdots *\!\!$
 \end{tabular}\right],
\qquad
0^{\leftarrow}:=\left[
\begin{tabular}{cc}
$*$& \\[-2.5mm]
$\vdots$&\!\!\!\Large 0\\[-2mm]
$*$&
\end{tabular}
\right]
\end{equation*}
in which all entries are zeros except
the last row of $0^{\downarrow}$ and
the first column of $0^{\leftarrow}$ that consist of stars.

The following theorem was proved by V.I. Arnold.

\begin{theorem}[{\cite[Theorem 4.4]{arn}}]\label{bhj}
Let
\begin{equation}\label{kus}
J:=\bigoplus_{i=1}^t \left(J_{m_{i1}}(\lambda _i)\oplus
J_{m_{i2}}(\lambda _i)\oplus\dots\oplus
J_{m_{ik_i}}(\lambda _i)\right)
\end{equation}
be a Jordan matrix of size $n\times n$, in which
\begin{equation}\label{kfu}
m_{i1}\ge
m_{i2}\ge\dots\ge m_{ik_i}\quad\text{for all }i
\end{equation}
and $\lambda _1,\dots,\lambda _t$ are distinct complex numbers. Then

\begin{itemize}
  \item[\rm{(a)}] {all complex matrices $J+X$
that are sufficiently close to $J$
can be simultaneously reduced by some
transformation
\begin{equation}\label{tef}
J+X\mapsto {\mathcal
S}(X)^{-1} (J+X) {\mathcal
S}(X),\quad\begin{matrix}
\text{${\mathcal S}(X)$
is nonsingular and}\\
\text{analytic at zero, } {\mathcal
S}(0)=I_n
\end{matrix}
\end{equation}
to the form
\begin{equation}\label{gep}
J+{\mathcal D}:=\bigoplus_{i=1}^t
\begin{bmatrix}
J_{n_{i1}}(\lambda _i)+0^{\downarrow}
&0^{\downarrow}&\dots&0^{\downarrow}
                \\
0^{\leftarrow}&
J_{n_{i2}}(\lambda _i)+0^{\downarrow}
&\dddots&\vdots
                \\
\vdots&\dddots&\dddots&0^{\downarrow}
                \\
0^{\leftarrow}&\dots&0^{\leftarrow}
&J_{n_{ik_i}}(\lambda _i)+0^{\downarrow}
\end{bmatrix}
\end{equation}
in which the stars of $\mathcal D$ represent elements
that depend analytically on the
entries of $X$;}

  \item[\rm{(b)}] {the number of stars in
${\mathcal D}$ is minimal that can be
achieved by transformations of the form
\eqref{tef}; this number of stars is equal to
the codimension of the similarity class
of the matrix $J$.}
\end{itemize}
\end{theorem}

\begin{example}\label{mwt}
If $J=J_3(5)\oplus J_2(5)$, then \eqref{gep} takes the form
\[
J+{\mathcal D}=\left[
\begin{array}{ccc|cc}
5 &1&0&0&0 \\
0&5 &1&0&0\\
0& 0&5&0& 0\\\hline
0 &0&0&5&1 \\
0&0&0&0& 5
\end{array}
\right]+
\left[
\begin{array}{ccc|cc}
0&0&0&0&0 \\
0&0&0&0&0\\
*& * &*&*& *\\\hline
* &0&0&0&0 \\
*&0&0&*&*
\end{array}
\right].
\]
\end{example}

\begin{remark}
Garcia-Planas and Sergeichuk
\cite{gar-ser} proved that the
statements (a) and (b) are also true for matrices
over $\mathbb R$ if the real Jordan
canonical form is taken instead of
\eqref{kus}, which simplifies the
miniversal deformations of real
matrices constructed by Galin
\cite{gal} (see also \cite[\S\,30E]{arn3}).
\end{remark}

\begin{remark}
Belitskii \cite{bel} proved that for each Jordan matrix $J$ there exists a permutation of rows and the same permutation of columns such that the obtained matrix $J'=S^{-1}JS$ ($S$ is a permutation matrix) possesses the property: all commuting with $J'$ matrices are upper block triangular. Klimenko and Sergeichuk \cite{K-S}  showed that the same permutations of rows and columns of \eqref{gep} transforms it to $J'+\mathcal D'=S^{-1}(J+\mathcal D)S$, which is lower block triangular. The matrix $J'$ was called the \emph{Weyr canonical matrix} by Sergeichuk \cite{ser}.
\end{remark}

\begin{remark}
Let $\mathcal D$ be the matrix in \eqref{gep}.
Denote by ${\mathcal D}({\mathbb C})$ the
vector space of all matrices obtained
from $\mathcal D$ by replacing its stars by
complex numbers.
Arnold \cite{arn} states that this vector space is a direct complement of the tangent space $T(J)$; that is,
\begin{equation}\label{jyr}
{\mathbb C}^{n\times
n}=T(J) \oplus {\mathcal
D}({\mathbb C}).
\end{equation}
(Thus, the number of
stars in $\mathcal D$ is equal to the
codimension of the similarity class of
$J$.) Moreover,
\emph{if ${\mathcal D}$ is any matrix
consisting of 0's and $*$'s that
satisfies \eqref{jyr}, then ${\mathcal D}$ can be used in
\eqref{gep}}.
This result about the similarity action of the group of nonsingular complex matrices was
generalized by Tannenbaum \cite[Part V,
Theorem 1.2]{tan} to a Lie group acting
on a complex manifold. Simplest
miniversal deformations of matrix
pencils and contagredient matrix
pencils \cite{kag,gar-ser}, and
matrices under congruence and
*congruence \cite{f_s,def-sesq} were
constructed by methods that are based
on direct sum decompositions analogous
to \eqref{jyr}.
\end{remark}

\subsection{An absolute value}

We extend Theorem \ref{bhj} to matrices over a field
$\F$ with topology given by an
\emph{absolute value}, which is a real
valued function $|.|:\mathbb F\to
\mathbb R$ such that
\begin{itemize}
  \item $|x|\ge 0$,
      $|x|=0\Leftrightarrow x=0$,
  \item $|xy|=|x||y|$, and
      $|x+y|\le |x|+|y|$
\end{itemize}
for all $x,y\in\mathbb F$; see
\cite[Chapter XII]{lang}. Then
$|1|=1$, $|-x|=|x|$, and $|x^{-1}|=|x|^{
-1}$.

A sequence $a_1,a_2,\dots$ in $\F$
\emph{converges} to $a \in \F$ if for
every $\varepsilon > 0$ there exists a
natural number $N$ such that $|a_n - a|
< \varepsilon $ for all $n\ge N$. A
sequence $a_1,a_2,\dots$  in $\F$ is a
\emph{Cauchy sequence} if for every
$\varepsilon> 0$, there exists $N$ such
that $|a_n -a_m | < \varepsilon $ for
all $m,n>N$. Every convergent sequence
is Cauchy. If every Cauchy sequence
converges, then $\F$ is called
\emph{complete}.

Each field $\F$ possesses the trivial
absolute value defined by $|0|=0$ and
$|x|=1$ for all $0\ne x\in \F$. By
\cite[Proposition 2.1]{lang}, each
field with nontrivial absolute value is
a dense subfield of a complete field,
which is unique up to $\F$-isomorphism.

\begin{example}\label{mft}
The most known complete fields with
nontrivial absolute value are
\begin{itemize}
\item the fields $\mathbb R$ and
    $\mathbb C$ with the usual
    absolute value,

  \item the field $\mathbb Q_p$ of $p$-adic
      numbers for any prime $p$ with
      the absolute value
\begin{equation}\label{kud}
|x|_p:=p^{-z},\qquad x=a_zp^z+a_{z+1}p^{z+1}+
\cdots\in\mathbb Q_p
\end{equation}
in which $z\in\mathbb Z$,
$a_{z},a_{z+1},\ldots\in\{0,1,\dots,p-1\}$
and $a_z\ne 0$,

\item the field $\F((T))$ of Laurent series
over a field $\F$
    with the absolute value
\begin{equation}\label{kye}
|x|:=2^{-z},\qquad x=a_zT^z+a_{z+1}T^{z+1}
+a_{z+2}T^{z+2}+
\cdots
\end{equation}
in which  $z\in\mathbb Z$,
$a_z,a_{z+1},a_{z+2},\ldots\in\mathbb
F$ and $a_z\ne 0$ (see \cite[p.
316]{cop}).
\end{itemize}

\end{example}

An absolute value is
\emph{non-Archimedean} if $|x+y|\le
\max\{|x|,|y|\}$ for all $x,y\in\mathbb
F$ (then $|x+y|= \max\{|x|,|y|\}$ if
$|x|\ne |y|$). Its geometric properties
are unaccustomed: every triangle is
isosceles (since
$|x-z|=|(x-y)+(y-z)|=\max\{|x-y|,|y-z|\}$
if $|x-y|\ne|y-z|$); every point of the
open sphere S$_{\varepsilon
}(x):=\{y\in\F|\,|y-x|<\varepsilon \}$
may serve as a center; any two spheres
are either disjoint or one is contained
inside the other; see \cite{nat}. The
absolute values \eqref{kud} and
\eqref{kye} are non-Archimedean.

\subsection{Real Jordan and Frobenius canonical matrices}

Define the $m\times m$ matrix
\[
C_m(a,b):=  \begin{bmatrix}
C(a,b)&I_2&&0  \\&C(a,b)& \ddots\\&&\ddots&I_2\\
0&&&C(a,b)
\end{bmatrix}\ \text{ with }
C(a,b):=\begin{bmatrix}
a&b\\-b&a\end{bmatrix}
\]
for each $a,b\in\mathbb R$ and $m=2,4,6,\dots.$
We use the \emph{real Jordan canonical form} (see \cite[Theorem 3.4.1.5]{HJ13}): each square real matrix is similar over $\mathbb R$ to a direct sum
\begin{equation}\label{ffg}
\begin{split}
C:=&\left[\bigoplus_{i=1}^{t'} \left(J_{m_{i1}}(\lambda _i)\oplus
J_{m_{i2}}(\lambda _i)\oplus\dots\oplus
J_{m_{ik_i}}(\lambda _i)\right)\right]
     \\
&\oplus\left[\bigoplus_{i=t'+1}^{t}
\left(C_{m_{i1}}(a_i,b_i)\oplus
C_{m_{i2}}(a_i,b_i)\oplus\dots\oplus
C_{m_{i{k_i}}}(a_i,b_i)\right)\right]
\end{split}
\end{equation}
with $\lambda _i,a_i,b_i\in\mathbb R$ and $b_i>0$; this direct sum is uniquely
determined up to
permutation of summands.
We suppose that
\begin{equation}\label{kff}
m_{i1}\ge m_{i2}\ge\dots\ge m_{ik_i}\qquad\text{for all }i.
\end{equation}

We also use the \emph{Frobenius canonical
form} for similarity (see \cite[Section
14]{pra}): each square matrix over an
arbitrary field $\F$ is similar to a
direct sum, determined uniquely up to
permutation of summands, of matrices of
the form
\begin{equation}\label{ktw}
\Phi_m (p):=\begin{bmatrix}
 0&1 &&0 \\&\ddots&\ddots\\0&&0&1\\
 -c_m&\dots&-c_2&-c_1
\end{bmatrix}\qquad
               (m\text{-by-}m)
\end{equation}
whose characteristic polynomial
$x^m+c_1x^{m-1}+\dots+c_m\in\F[x]$ is
an integer power of a polynomial $p(x)$ that
is irreducible over $\F$.

\section{The main result}

Let us fix a field $\F$ that is complete with
respect to a nontrivial absolute value, and a matrix $A\in\F^{n\times n}$.  In this section, we formulate Theorem \ref{teo2} about miniversal deformations of $A$ over $\F$.

Let ${\mathcal
D}$ be an $n\times n$ matrix consisting
of 0's and $*$'s such that
\begin{equation}\label{jyr1d}
{\mathbb F}^{n\times
n}=T(A) \oplus {\mathcal
D}({\mathbb F}),
\end{equation}
in which $T(A)$ is defined by
\eqref{eelie} with $\mathbb F$ instead of $\mathbb C$, and ${\mathcal D}({\mathbb
F})$ is the vector space of all
matrices obtained from $\mathcal D$ by
replacing its stars with elements of
$\F$. Such a matrix ${\mathcal D}$ always
exists; to construct it we can take the
set of matrix units lexicographically
arranged:
\[
E_{11},\ E_{12},\ \dots,\ E_{1n};\
E_{21},\ E_{22},\ \dots,\ E_{2n};\
\dots;\
E_{n1},\ E_{n2},\ \dots,\ E_{nn},\
\]
and delete those of them that are
linear combinations of the preceding
units and elements of $T(A)$. The
remaining units generate a vector space that
can be used as ${\mathcal D}({\mathbb F})$.

Let us fix $n^2$ matrices $F_{ij}\in\mathbb
F^{n\times n}$, $i,j=1,\dots,n$, such that
\begin{equation}\label{8}
E_{ij}+F_{ij}A
-AF_{ij}\in {\mathcal
D}({\mathbb F})
\end{equation}
for each $n\times n$ matrix unit
$E_{ij}$ ($F_{ij}$ exists by \eqref{jyr1d}).
We can and will take
\begin{equation}\label{mrk}
F_{ij}=0_n\qquad \text{if $E_{ij}\in {\mathcal
D}({\mathbb F})$}.
\end{equation}
Define the neighborhood of zero \begin{equation}\label{kuxd}
U:=\left\{X\in\F^{n\times n}\biggm|
\|X\|<\frac{1}{48\sqrt{n}(a+1)f^2} \right\},
\end{equation}
in which
\begin{equation}\label{kux}
a:=\|A\|,\qquad f:=\max\bigm\{
\sum_{i,j}\|F_{ij}\|,\textstyle\frac13\bigm\},
\end{equation}
and
\[
\|M\|:=\sqrt{\sum |m_{ij}|^2}\qquad \text{for all $M=[m_{ij}]\in\mathbb
      F^{n\times n}$.}
\]

For each $X\in U$, we construct a
sequence
\begin{equation}\label{rtg}
M_1:=X,\ M_2,\ M_3,\dots
\end{equation}
of $n\times n$ matrices as follows: if
\begin{equation}\label{mro}
M_k=[m_{ij}^{(k)}]
\end{equation}
 has been
constructed, then $M_{k+1}$ is defined
by
\begin{equation}\label{dei}
A+M_{k+1}:=(I_n-C_k)^{-1}(A+M_k)(I_n-C_k),
\quad
C_k:=\sum_{i,j}m_{ij}^{(k)}F_{ij}.
\end{equation}

Our main result is the following theorem.

\begin{theorem}
\label{teo2}

Let $\F$ be a complete field with
respect to a nontrivial absolute value, and let $A\in\F^{n\times n}$. Let ${\mathcal
D}$ be an $n\times n$ matrix consisting
of 0's and $*$'s that satisfies \eqref{jyr1d}. Then the following statements hold:

\begin{itemize}
  \item[\rm(a)] Let $C_1,C_2,\dots$ be formed by \eqref{dei}.  For each matrix $X$ from the set
      $U$ defined in \eqref{kuxd}, the infinite
      product
\begin{equation}\label{gre}
{\mathcal
S}(X):=(I_{n}-C_1)(I_{n}-C_2)(I_{n}-C_3)\cdots
\end{equation}
is convergent and nonsingular. The related matrix
function ${\mathcal S}:U\to \mathbb
F^{n\times n}$ is continuous and ${\mathcal
S}(0_n)=I_n$; this function is analytic if $\F=\mathbb C$.

\item[\rm(b)] If $X\in U$, then
\begin{equation}\label{msu1}
D(X):={\mathcal
S}(X)^{-1}(A+X){\mathcal S}(X)-A\in {\mathcal
D}(\mathbb F),
\end{equation}
for which ${\mathcal S}(X)$ is
    defined in \eqref{gre}. This means that all matrices $A+X$
    with $X\in U$ are reduced by
    the similarity transformation
    ${\mathcal S}(X)^{-1}(A+X) {\mathcal
    S}(X)$ to the
    form $A+{\mathcal D}$. The
stars in ${\mathcal D}$ represent the entries that
depend continuously on the entries
of $X$. The number of stars in
${\mathcal D}$ is equal to the
codimension of the similarity class
of $A$.

In particular, for each $\varepsilon$ satisfying $0<\varepsilon\le 1/2$ and for each $X\in 2\varepsilon U$, we have
\begin{align}\label{geo}
\|{\mathcal S}(X)-I_n\|&<
-1+(1+\varepsilon) (1+\varepsilon
^2) (1+\varepsilon ^3)\cdots,\\
\label{geok} \|D(X)\|&\le
\varepsilon/(2f),
\end{align}
in which $f$ is defined in \eqref{kux}.

\item[\rm(c)] Let one of the following conditions hold:
\begin{itemize}
  \item[\rm(i)] $A$ is a Frobenius canonical matrix
\begin{equation*}\label{klu}
\bigoplus_{i=1}^t
\left(\Phi_{m_{i1}}(p_i)\oplus
\Phi_{m_{i2}}(p_i)\oplus\dots\oplus
\Phi_{m_{ik_i}}(p_i)\right),\quad
m_{i1}\ge\dots\ge m_{ik_i}
\end{equation*}
$($see \eqref{ktw}$)$ in which $p_1,p_2,\dots,p_t$ are
distinct irreducible polynomials
over $\F$,

  \item[\rm(ii)] $\F=\mathbb C$ and
    $A$ is a Jordan
canonical matrix \eqref{kus} satisfying \eqref{kfu},

  \item[\rm(iii)] $\F=\mathbb R$ and $A$ is a real Jordan canonical matrix \eqref{ffg} satisfying \eqref{kff}.
\end{itemize}
Then the matrix $\mathcal D$ satisfying \eqref{jyr1d} can be taken as follows:
\begin{equation*}\label{ged1}
{\mathcal D}:=\bigoplus_{i=1}^t
\begin{bmatrix}
0_{m_{i1}}^{\downarrow}
&0^{\downarrow}&\ \ \dots\ \ &0^{\downarrow}
                \\
0^{\leftarrow}&
0_{m_{i2}}^{\downarrow}
&\ddots&\vdots
                \\
\vdots&\ddots&\ddots&0^{\downarrow}
                \\
0^{\leftarrow}&\dots&0^{\leftarrow}
&0_{m_{ik_i}}^{\downarrow}
\end{bmatrix},
\end{equation*}
in which $0_m^{\downarrow}$ denotes the matrix $0^{\downarrow}$ of size $m\times m$.
\end{itemize}
\end{theorem}

Note that Theorem \ref{teo2}(a) constructs a
matrix ${\mathcal S}(X)$ that transforms a family of all matrices $A+X$ with $X\in U$ to the form $A+{\mathcal D}$. Garcia-Planas and Mailybaev
\cite{gar_mai,mai2000,mai2001}
construct
analogous matrices (in the
form of Taylor series) that transform families of complex matrices under
similarity and families of complex matrix
pencils under strict equivalence to their miniversal
deformations. They also give numerous
applications.

\section{Proof of the main result} \label{sect4}

In the remaining portion of the article we prove parts (a)--(c) of Theorem
\ref{teo2}.

\subsection{Proof of part (a)} \label{sect4a}

By \eqref{8}, \eqref{mro}, and
\eqref{dei},
\begin{align*}\label{8a}
\sum_{i,j}m_{ij}^{(k)} E_{ij}+
\sum_{i,j}m_{ij}^{(k)}F_{ij}A -
\sum_{i,j}m_{ij}^{(k)}AF_{ij}&\in {\mathcal
D}({\mathbb F}),\\
M_k+C_kA-AC_k&\in{\mathcal D}({\mathbb F}).
\end{align*}

For each $P=[p_{ij}]\in\mathbb
      F^{n\times n}$, we write
\[
\|P\|_{\mathcal
D}:=\sqrt{\sum_{
{(i,j)\notin{\mathcal
I}}(\mathcal D)} |p_{ij}|^2},
\]
in which ${\mathcal I}({\mathcal D})\subseteq
\{1,\dots,n\}\times \{1,\dots,n\}$ is
the set of indices of the stars in
${\mathcal D}$.

Let us fix $\varepsilon\in\mathbb R$ such that
$0<\varepsilon \le 1/2$. Define a
sequence
\begin{equation*}\label{21z}
\delta_1,\
\tau_1,\
\delta_2,\
\tau_2,\
\delta_3,\
\tau_3,\ \dots
\end{equation*}
of positive real numbers by induction:
\begin{gather}\nonumber
\tau_1=\delta_1:=\varepsilon /(8
fv)=\varepsilon /(24
\sqrt{n}(a+1)f^2),\\\label{lob}
\tau_{k+1}:=\tau_k+\delta_kv ,\quad
\delta_{k+1}:=\delta
_k\varepsilon=\delta _1\varepsilon^{k}
\quad(k=1,2,\dots),
\end{gather}
in which $a$ and $f$ were defined in \eqref{kux} and
\[
v:=3\sqrt{n}(a+1)f.
\]

\begin{lemma}\label{lib}
If $\|M_1\|<\tau _1$ in \eqref{rtg},
then
\begin{equation}\label{hte}
\|M_{\ell}\|<\tau_{\ell}
,\quad
\|M_{\ell}\|_{\mathcal D} <\delta_{\ell},\quad
\|C_{\ell}\|\le \delta _{\ell}f
\quad(\ell=1,2,\dots).
\end{equation}
\end{lemma}

\begin{proof}
If the first inequality in \eqref{hte}
holds, then the third holds too since
\begin{align*}\label{mmfd}
 \|C_{\ell}\|&=\|\!\!\sum_{
{(i,j)\notin{\mathcal
I}}(\mathcal D)}m_{ij}^{({\ell})}F_{ij}\|
\quad (\text{by \eqref{mrk} and \eqref{dei})}
\\&\le
\sum_{
{(i,j)\notin{\mathcal
I}}(\mathcal D)}|m_{ij}^{({\ell})}|\cdot
\|F_{ij}\|
\le \sum_{
{(i,j)\notin{\mathcal
I}}(\mathcal D)}\delta _{\ell}\cdot
\|F_{ij}\|
\le \delta _{\ell}f<1/10.
\end{align*}

The series
\[
(I-C_{\ell})^{-1}=I+C_{\ell}+C_{\ell}^2+\cdots
\]
is convergent since the sequence of its
partial sums $I$, $I+C_{\ell}$,
$I+C_{\ell}+C_{\ell}^2,\ \dots$ is a
Cauchy sequence. Moreover,
\begin{align*}
\|(I-C_{\ell})^{-1}\|&=
\|I+C_{\ell}+C_{\ell}^2+\cdots\|\le
\|I\|+\|C_{\ell}\|+\|C_{\ell}^2\|+\cdots
\\&\le \sqrt{n}(1+\|C_{\ell}\|
+\|C_{\ell}\|^2+\cdots)
=\sqrt{n}/(1-\|C_{\ell}\|)<1.5\sqrt{n}.
\end{align*}

Suppose that \eqref{hte} holds for
all $\ell\le k$; let us
 prove it for $\ell=
k+1$. Write $A_k:=A+M_k$, then
\begin{equation*}\label{geo1}
\|A_k\|=\|A+M_k\|\le\|A\|+\|M_k\|
\le a+\delta _k<a+1.
\end{equation*}
By \eqref{dei},
\[
A_{k+1}=(I+C_k+C_k^2+\cdots)A_k(I-C_k)
=A_k+(I-C_k)^{-1}(C_kA_k-A_kC_k).
\]
Subtracting $A$ from these equations and taking the absolute
value, we get
\begin{align*}
 \|M_{k+1}\|\le \|M_{k}\|+1.5\sqrt{n}\cdot
2\|C_k\|\cdot\|A_k\|
 \le
\tau _k+3\sqrt{n}\cdot\delta _kf\cdot
(a+1)=\tau _{k+1},
\end{align*}
which proves the first inequality in
\eqref{hte} for $\ell= k+1$.

Let us prove the second inequality  in \eqref{hte}. By
\eqref{kkd},
\[
A+M_{k+1}=A+M_k+C_k(A+M_k)-(A+M_k)C_k
+C_k(I-C_k)^{-1}(C_kA_k -A_kC_k),
\]
\[
M_{k+1}=\underbrace{M_k+C_kA-A_kC_k}
_{\text{belongs to ${\mathcal
D}({\mathbb F})$}}
 +C_kM_k-M_kC_k
+C_k(I-C_k)^{-1}(C_kA_k -A_kC_k),
\]
\begin{align*}
\|M_{k+1}\|_{\mathcal
D}&\le 2\|C_k\|\cdot\|M_k\|
+\|C_k\|\cdot 1.5\sqrt{n}\cdot 2\|C_k\|
\cdot\|A_k\|
\\&\le 2\cdot \delta _kf\cdot(\tau _k+
1.5\sqrt{n}\cdot\delta_k f\cdot(a+\delta _k))
\le \delta_k f(2\tau _k+v\delta _k).
\end{align*}
This reduces the request inequality $\|M_{k+1}\|_{\mathcal
D}
<\delta_{k+1}(=\delta_{k}\varepsilon)$
to the inequality
$2f\tau _k+fv\delta _k\le \varepsilon
$.

Since
\begin{equation}\label{kku}
\begin{split}
\tau _k&= \tau _{k-1}+v\delta_{k-1}
= \tau _{k-2}+v(\delta_{k-2}+\delta_{k-1})
\\&\le \tau _1+v(\delta_1+\delta_2+\delta_3+
\cdots)=\delta_1+v(\delta_1
+\delta_1\varepsilon +\delta_1\varepsilon^2+
\cdots)
\\&=\delta_1(1
+v/(1-\varepsilon))\le
\delta_1(1+2v)\le 3\delta_1v,
\end{split}
\end{equation}
we obtain that
\[
2f\tau _k+fv\delta _k\le 6f\delta_1v+fv\delta _k
\le 7f\delta_1v\le 7\varepsilon /8\le \varepsilon.
\qedhere
\]
\end{proof}

\begin{proof}[Proof of Theorem
\ref{teo2}(a)] Write
\begin{equation}\label{llv}
{\mathcal S}_{k,l}(X):=\prod_{i=k}^{l}
(I-C_i),\quad {\mathcal S}_{l}(X):=
{\mathcal S}_{1,l}(X)\qquad\text{for all
$1\le k\le l\le \infty$}.
\end{equation}

If $\|X\|\le \tau _1= \varepsilon /(8
fv)$ and $l<\infty$, then by Lemma
\ref{lib}
\begin{align}\nonumber
\|{\mathcal
S}_{k,l}(X)-I\|&=
 \|(I-C_k)(I-C_{k+1})
\cdots(I-C_l)-I\|
  \\\nonumber&\le\sum_{k\le i\le l}\|C_i\|
+\sum_{k\le i<j\le l}\|C_i\|\,\|C_j\|+\cdots
+\|C_k\|\cdots \|C_l\|
  \\\nonumber&\le -1+(1+\|C_k\|)(1+\|C_{k+1}\|)\cdots
(1+\|C_{l}\|)
  \\\nonumber&\le -1+(1+ \varepsilon ^{k-1}\delta_1f)
(1+\varepsilon^{k}\delta_1f)
(1+\varepsilon^{k+1}\delta_1f)\cdots
  \\\nonumber& \le {\textstyle -1+\left(1+  \frac{\varepsilon^{k}}{8v}\right)
\left(1+\frac{\varepsilon^{k+1}}{8v}\right)
\left(1+\frac{\varepsilon^{k+2}}{8v}\right)\cdots}
  \\\label{lfh}&\le -1+(1+\varepsilon^{k})
(1+\varepsilon^{k+1})
(1+\varepsilon^{k+2})\cdots
\end{align}
The infinite product
$\omega_k(\varepsilon )
:=\prod_{i=k}^{\infty}(I+\varepsilon
^i)$ is convergent since any product
$\prod_{i=1}^{\infty}(1+a_i)$ with
positive real $a_i$'s converges if and
only if the infinite series
$\sum_{i=1}^{\infty}a_i$ converges (see
\cite[Theorem 15.14]{mark}). Write
$\omega_k:=\omega_k(1/2)$.

The
infinite product ${\mathcal
S}(X)=\prod_{i=1}^{\infty}(I-C_i)$
defined in \eqref{gre} is convergent for each $X\in U$.
Indeed, the sequence ${\mathcal
S}_{1}(X)$, ${\mathcal S}_{2}(X)$, ${\mathcal
S}_{3}(X),\dots$ is a Cauchy sequence
since $1\le k\le l$ implies
\begin{align*}
\|{\mathcal S}_{l}(X)&-{\mathcal S}_{k}(X)\|
=\|{\mathcal S}_{k}(X)({\mathcal S}_{k+1,l}(X)-I)\|
   \\&\le \|({\mathcal S}_{k}(X)-I)+I\|\cdot
\|{\mathcal S}_{k+1,l}(X)-I)\|
   \\&\le ((\omega _1-1)+\sqrt{n})
   (\omega _{k+1}-1) \quad\text{by \eqref{lfh}}.
\end{align*}

Let us prove that the matrix function
${\mathcal S}:U\to \mathbb F^{n\times n}$
is continuous. By \eqref{dei}, the
entries of each $C_i$ are polynomials
in the entries of $X$. Thus, the
entries of ${\mathcal S}_k(X)$ are
polynomials in the entries of $X$ for
each $k$. Since the operations $x+y$,
$x-y$, $xy$, $x^{-1}$ are continuous,
the matrices ${\mathcal S}_k(X)$ are
continuous functions. Let
$\varepsilon'$ be a small positive
number. For any $X,Y\in U$ and $k\in\mathbb
N$, we have
\begin{align*}
{\mathcal S}(Y)&-{\mathcal S}(X)=
 {\mathcal S}_{k}(Y){\mathcal S}_{k+1,\infty}(Y)
 -{\mathcal S}_{k}(X){\mathcal S}_{k+1,\infty}(X)
\\&=({\mathcal S}_{k}(Y)
-{\mathcal S}_{k}(X)){\mathcal S}_{k+1,\infty}(Y)
+{\mathcal S}_{k}(X)({\mathcal S}_{k+1,\infty}(Y)
 -{\mathcal S}_{k+1,\infty}(X)),
\end{align*}
hence
\begin{align*}
&\|{\mathcal S}(Y)-{\mathcal S}(X)\|
    \\&\le\|{\mathcal S}_{k}(Y)
-{\mathcal S}_{k}(X)\|
\cdot\|{\mathcal S}_{k+1,\infty}(Y)\|
+\|{\mathcal S}_{k}(X)\|\cdot\|{\mathcal S}_{k+1,\infty}(Y)
 -{\mathcal S}_{k+1,\infty}(X)\|
    \\&\le\|{\mathcal S}_{k}(Y)
-{\mathcal S}_{k}(X)\|
(\omega _1+\sqrt{n})
+\omega _1\|{\mathcal S}_{k+1,\infty}(Y)
 -{\mathcal S}_{k+1,\infty}(X)\|.
\end{align*}
Because ${\mathcal S}_{k}(X)$ converges
uniformly, there exists $\delta '>0$
such that
\[\|{\mathcal S}_{k}(Y)
-{\mathcal S}_{k}(X)\|\cdot
(\omega _1+\sqrt{n})<\varepsilon'/2\qquad
\text{if $\|Y-X\|< \delta '$}.\] Since
$\omega _t\to 1$, there exists $k$ such
that
\begin{align*}
\omega _1 \|{\mathcal S}_{k+1,\infty}(Y)
 &-{\mathcal S}_{k+1,\infty}(X)\|
   \le \omega _1\big(\|{\mathcal S}_{k+1,\infty}(Y)-I\|
 +\|{\mathcal S}_{k+1,\infty}(X)-I\|\big)
\\&\le \omega _1\cdot 2(\omega _{k+1}-1)
<\varepsilon'/2.
\end{align*}
Thus, $\|{\mathcal S}(Y)-{\mathcal S}(X)\|<
\varepsilon'$ if $\|Y-X\|< \delta '$
and so ${\mathcal S}:U\to \mathbb
F^{n\times n}$ is a continuous
function.

The matrix ${\mathcal S}(X)$ is
nonsingular for each $X\in U$. Indeed,
${\mathcal S}(X)
 ={\mathcal S}_{k}(X){\mathcal S}_{k+1,\infty}(X)$, in which
${\mathcal S}_{k}(X)$ is nonsingular for
each $k$, and ${\mathcal
S}_{k+1,\infty}(X)$ is certainly
nonsingular for a sufficiently large
$k$ since $\|{\mathcal
S}_{k+1,\infty}(X)-I\|\to 0$ by
\eqref{lfh}.

If $X=0_n$, then all $M_i=C_i=0_n$, and
so ${\mathcal S}(0_n)=I_n$.

Let now $\F=\mathbb C$; we should prove that ${\mathcal S}:U\to \mathbb
F^{n\times n}$ is analytic.
By \eqref{dei},
the entries of each $C_i$ are
polynomials in the entries of $X$.
Hence the entries of each $ {\mathcal
S}_{\ell }(X)$ (see \eqref{llv}) are
polynomials in the entries of $X$.
Since ${\mathcal S}_{\ell }(X)\to {\mathcal
S}(X)$, the Weierstrass theorem on
uniformly convergent sequences of
analytic functions (see \cite[Theorem
15.8]{mark}) ensures that the entries of
${\mathcal S}(X)$ are analytic functions
in the entries of $X$.
\end{proof}

\subsection{Proof of part (b)} \label{seke}

By \eqref{dei} and \eqref{llv},
\[A+M_k(X)={\mathcal S}_{k-1}(X)^{-1}
(A+X){\mathcal S}_{k-1}(X)\quad\text{for each }X\in U.
\]
Hence, $M_k(X)\to D(X)$  as $k\to
\infty$, where $D(X)$ is defined in
\eqref{msu1}. By \eqref{hte} and
\eqref{lob}, $\|M_{k}(X)\|_{\mathcal D}
<\delta_{k}= \delta
_1\varepsilon^{k-1}\to 0$ as $k\to
\infty$, and so $\|D(X)\|_{\mathcal D}=0$,
which proves \eqref{msu1}. The
inequality \eqref{geo} follows from
\eqref{lfh} and the inequality
\eqref{geok} follows from
\begin{align*}
\|M_k\|&\le
\tau_k\le 3\delta_1v \quad(\text{by }\eqref{kku})\\
&\le 3\varepsilon/(8f)
\le  \varepsilon/(2f).
\end{align*}
Since \eqref{jyr1d} is a direct sum,
the number of stars in ${\mathcal D}$ is
equal to $n^2-\dim T(A)$, which is the
codimension of the similarity class of
$A$.

\subsection{Proof of part (c)} \label{sekh}

Let us prove the statement (i) in (c).
Write
\begin{equation*}\label{eeg}
T(P,Q):=\{XQ-PX\,|\,X\in{\mathbb
F}^{m\times n}\},\qquad P\in\F^{m\times m},\
Q\in\F^{n\times n}.
\end{equation*}
Let $A=A_1\oplus\dots\oplus
A_l\in\F^{n\times n}$, in which every
$A_i$ is of size $n_i\times n_i$. Let ${\mathcal
D}=[{\mathcal D}_{ij}]_{i,j=1}^l$ be a block
matrix, in
which every $\mathcal D_{ij}$ is a block of size $n_i\times
n_j$ whose entries are 0's and $*$'s. Clearly, ${\F}^{n\times n}=T(A)\oplus
{\mathcal D}({\F})$ if and only if
\begin{equation*}\label{lka}
{\F}^{n_i\times
n_j}=T(A_i,A_j)\oplus {\mathcal D}_{ij}({\F})\qquad\text{for all }
i,j=1,\dots,l.
\end{equation*}
Thus, the statement (i) is implied
by the following lemma.

\begin{lemma}\label{lsb}
Let $\Phi=\Phi(p^r)$ and
$\Psi=\Phi(q^s)$ be $m\times m$ and
$n\times n$ Frobenius blocks $($see
\eqref{ktw}$)$ in which $p,q\in\F[x]$
are irreducible polynomials with
leading coefficient $1$. The following hold:
\begin{itemize}
  \item[$(\alpha)$] If $p\ne q$, then
      ${\F}^{m\times
n}=T(\Phi,\Psi)$.

  \item[$(\beta)$] If $p=q$ and $r\ge
      s$, then ${\F}^{m\times
n}=T(\Phi,\Psi)\oplus
0^{\downarrow}(\F)$.

  \item[$(\gamma)$] If $p=q$ and
      $r\le s$, then ${\F}^{m\times
n}=T(\Phi,\Psi)\oplus
0^{\leftarrow}(\F)$.
\end{itemize}
\end{lemma}

\begin{proof}
Let us compute the dimension of
$T(\Phi,\Psi).$ The space
$T(\Phi,\Psi)$ is the image of the
linear operator
\begin{equation}\label{liu}
\xi: \F^{m\times n}\to
\F^{m\times n},\quad X\mapsto X\Psi-\Phi
X;
\end{equation}
hence
\begin{equation}\label{lku}
\dim T(\Phi,\Psi)=mn-\Ker(\xi),\quad
\Ker(\xi)= \{H\in \F^{m\times
n}\,|\,H\Psi=\Phi H\}.
\end{equation}

For each matrix $A\in\F^{t\times t}$,
denote by $M(A)=(\F^{t},A)$  the module over
the polynomial ring $\F[x]$ that is the vector space $\F^{t}$
with multiplication $f(x)v:=f(A)v$ for
all $f\in\F[x]$ and $v\in\F^{t}$. Each
matrix $H\in \Ker(\xi)$ defines the
$\F[x]$-homomorphism \[\varphi_H:M(\Psi)\to
M(\Phi),\quad v\mapsto Hv;\]
moreover, $H\mapsto \varphi_H$ is a
linear bijection from $\Ker(\xi)$ to
the $\mathbb F$-space of homomorphisms
$\Hom_{\F[x]}(M(\Psi), M(\Phi))$. Thus,
\[
\dim_{\F} \Ker(\xi)=
\dim_{\F} \Hom_{\F[x]}(M(\Psi), M(\Phi)).
\]
The module $M(\Psi)$ is cyclic because
\[M(\Psi)=\F[x]e_n\simeq
\F[x]/q(x)^n\F[x],\qquad\text{where }
e_n:=(0,\dots,0,1)^T\in \F^{n}.\] Hence, each homomorphism
$\varphi:M(\Psi)\to M(\Phi)$ is fully
determined by $\varphi(e_n)$.

\begin{itemize}
  \item If $p\ne q$, then $\varphi(e_n)=0$
since
\begin{equation}\label{ldr}
q(x)^s\varphi(e_n)
=\varphi(q(x)^se_n)=\varphi(0)=0,
\end{equation}
$p(x)^r\varphi(e_n)\in
p(x)^rM(\Phi)=0$, and $fp^r+gq^r=1$ for
some $f,g\in\F[x]$.

  \item If $p=q$  and $r\ge s$, then
$\varphi(e_n)\in p(x)^{r-s}M(\Phi)$ by
\eqref{ldr}. Moreover, $\varphi(e_n)$
is an arbitrary element of the submodule
$p(x)^{r-s}M(\Phi)\simeq M(\Psi)$ of dimension $n$.

  \item If $p=q$  and $r\le s$, then
$\varphi(e_n)$ is an arbitrary element of
$M(\Phi)$ of dimension $m$.

\end{itemize}
Therefore,
\[
\dim_{\F} \Ker(\xi)=
\left\{
  \begin{array}{ll}
    0 & \hbox{if $p\ne q$} \\
    n & \hbox{if $p=q$  and $r\ge s$} \\
    m & \hbox{if $p=q$  and $r\le s$}
  \end{array}
\right.
\]
for the linear operator \eqref{liu}. By \eqref{lku},
\begin{equation}\label{kve}
\dim_{\F} T(\Phi,\Psi)=
\left\{
  \begin{array}{ll}
    mn & \hbox{if $p\ne q$} \\
    (m-1)n & \hbox{if $p=q$  and $r\ge s$} \\
    m(n-1) & \hbox{if $p=q$  and $r\le s$}.
  \end{array}
\right.
\end{equation}

$(\alpha)$ If $p\ne q$, then ${\F}^{m\times
n}=T(\Phi,\Psi)$ by \eqref{kve}.

$(\beta)$ Let $p=q$ and $r\ge s$. Denote by
$E_{ij}$ the $(i,j)$ matrix unit
in $\F^{m\times n}$.
For each $i=1,\dots,m-1$, the matrices $\xi(-E_{i1})$,
$\xi(-E_{i2})$, \dots, $\xi(-E_{in})$
have the form
\begin{equation*}\label{mmkw}
\begin{bmatrix}
 0&0&\cdots&0\\[-7pt]\hdotsfor{4}\\
 0&0&\cdots&0\\ 1&0&\cdots&0\\
*&*&\cdots&*\\[-6pt]\hdotsfor{4}\\ *&*&\cdots&*\\
\end{bmatrix},\
\begin{bmatrix}
 0&0&\cdots&0\\[-7pt]\hdotsfor{4}\\
 0&0&\cdots&0\\ 0&1&\cdots&0\\
*&*&\cdots&*\\[-6pt]\hdotsfor{4}\\ *&*&\cdots&*\\
\end{bmatrix},\,\ \dots,\
\begin{bmatrix}
 0&0&\cdots&0\\[-7pt]\hdotsfor{4}\\
 0&0&\cdots&0\\ 0&0&\dots&1\\
*&*&\cdots&*\\[-6pt]\hdotsfor{4}\\ *&*&\cdots&*\\
\end{bmatrix},
\end{equation*}
respectively,
with units in the $i$th row.
Since $T(\Phi,\Psi)=\im \xi$, we have ${\F}^{m\times
n}=T(\Phi,\Psi)+0^{\downarrow}(\F)$.
This sum is direct by \eqref{kve} and
because $\dim 0^{\downarrow}(\F)=n$.

$(\gamma)$ Let $p=q$ and $r\le s$.
For each $j=2,\dots,n$, the matrices $\xi(E_{1j})$,
$\xi(E_{2j})$, \dots, $\xi(E_{mj})$
have the form
\begin{equation*}\label{mdkw}
\begin{bmatrix}
 *&\cdots&*&1&0&\cdots&0
 \\
*&\cdots&*&0&0&\cdots&0\\[-6pt]\hdotsfor{7}\\ *&\cdots&*&0&0&\cdots&0
\end{bmatrix},\,\ \dots,\
\begin{bmatrix}
 *&\cdots&*&0&0&\cdots&0
 \\
*&\cdots&*&0&0&\cdots&0\\[-6pt]\hdotsfor{7}\\ *&\cdots&*&1&0&\cdots&0
\end{bmatrix}\end{equation*}
with units in the $j$th column.
Since $T(\Phi,\Psi)=\im \xi$, we have
${\F}^{m\times n}=T(\Phi,\Psi)+
0^{\leftarrow}(\F)$. This sum is direct
by \eqref{kve} and because $\dim
0^{\leftarrow}(\F)=m$.

We have proved the statement (i) of part (c). The statements (ii) and (iii) were proved by Arnold \cite[Theorem 4.4]{arn} and by Garcia-Planas and Sergeichuk \cite[Theorem 3.1]{gar-ser}, respectively.
\end{proof}

\section*{Acknowledgements}
The work on the final version of this article was supported in part by the UAEU UPAR grant G00001922. The authors would like to thank the referees for their constructive comments, which helped us to improve the manuscript.

\bibliographystyle{amsalpha}

\end{document}